\documentclass[reqno,12pt]{amsart}
\usepackage{amsmath,amssymb,amsthm,mathtools}
\usepackage{mathrsfs}
\usepackage{bm}
\usepackage{tikz-cd}
\usepackage{xcolor}
\usepackage{hyperref}
\usepackage{comment}
\usepackage{orcidlink}
\hypersetup{
    colorlinks=true,
    linkcolor=black,
    citecolor=black,
    urlcolor=black}
\usepackage[margin=1.2in]{geometry}
\newtheorem{theorem}{Theorem}[section]
\newtheorem{lemma}[theorem]{Lemma}
\newtheorem{proposition}[theorem]{Proposition}

\theoremstyle{definition}
\newtheorem*{thmA}{Theorem A}
\newtheorem{definition}[theorem]{Definition}

\allowdisplaybreaks

\newcommand{\Z}{\mathbb{Z}}
\newcommand{\Q}{\mathbb{Q}}

\newcommand{\FF}{{\mathcal{F}}}
\newcommand{\p}{\mathfrak{p}}

\newcommand{\OO}{\mathcal{O}}

\newcommand{\Gal}{\mathrm{Gal}}
\newcommand{\Hom}{\mathrm{Hom}}

\title{Euler characteristic of the Selmer group attached to an Artin representation}

\author{Subhasis Panda}

\address{}

\email{subhasispanda559@gmail.com}



\begin{document}

\begin{abstract} {Longo and Vigni extended a result of Greenberg by giving a relation between the cardinality of the Selmer group and the characteristic power series of its Pontryagin dual over the cyclotomic $\Z_p$-extension for a $p$-adic Galois representation. In this paper, we extend this result to the case of Artin representations, following the framework of Greenberg and Vatsal.
\medskip
}\end{abstract}

\maketitle

\tableofcontents
\section{Introduction}
Let $E$ be an elliptic curve defined on a number field $F$. 
Assume that $E$ has good ordinary reduction in all primes of $F$ lying over $p$. 
Furthermore, assume that the $p$-primary Selmer group $\text{Sel}_p(E/F)$ of $E$ over $F$ is finite. 
For a cyclotomic $\Z_p$-extension $F_\infty$ of $F$ with Galois group $\Gamma:=\Gal(F_\infty/F)$, let $\Lambda:= \Z_p[[\Gamma]]$ denote the Iwasawa algebra. 
Since $\text{Sel}_p(E/F)$ is finite, $p$-primary Selmer group $\text{Sel}_p(E/F_\infty)$ of $E$ over $F_\infty$ is $\Lambda$-torsion. 
If $\mathfrak{F}_E \in \Lambda$ is a generator of the characteristic ideal of the Pontryagin dual of $\text{Sel}_p(E/F_\infty)$, then Greenberg \cite[Theorem 4.1]{GR5} proved the following
\begin{equation} \label{eq.1}
    \mathfrak{F}_E(0) \sim \frac{\#\text{Sel}_p(E/F) \cdot \displaystyle \prod_{v \mid p} \# \big ( \Tilde{E_v}(\mathbb{F}_v)_p \big )^2 \cdot \prod_{v \, \text{bad}} c_v(E)}{\# \big (  E(F)_p \big )^2},
\end{equation}
where the symbol $\sim$ indicates that the two quantities have the same $p$-adic valuation, $c_v(E)$
is the Tamagawa number of $E$, $\Tilde{E_v}(\mathbb{F}_v)_p$ is the $p$-torsion points of the reduction $\Tilde{E_v}$ of $E$ on the residue field $\mathbb{F}_v$ of $F$ at a prime $v$ and $E(F)_p$ is the $p$-torsion points of the Mordell-Weil group $E(F)$. Formula (\ref{eq.1}) is a special case of a result proved by Perrin-Riou\cite{PerrinRiou1984} when $E$ has complex multiplication, and by Schneider\cite{Schneider1983} in the general case. \\

This result of Greenberg has been generalized in \cite{LV} for a more general $p$-adic Galois representation, including those coming from  $p$-ordinary modular forms of weight at least $4$.  Following \cite{LV}, let $F$  be a number field with absolute Galois group $G_F=\Gal (\overline{F}/F)$ and an odd prime $p$, consider a $p$-ordinary representation
\begin{equation*}
    \sigma: G_F \to \text{Aut}_K(V) \simeq \text{GL}_{r}(K),
\end{equation*}
where $V$ is a $r$-dimensional vector space over a finite extension $K$ of $\Q_p$. 
Consider a $G_F$-stable $\OO$-lattice $T$ of $V$ and set $A:=V/T$, where $\OO$ is the ring of integers of $K$. One can associate the Greenberg Selmer group $\text{Sel}_{\text{Gr}}(A/F_\infty)$  over $F_\infty$ and the Bloch-Kato Selmer group \cite[Section-3]{LV} $\text{Sel}_{\text{BK}}(A/F)$  over $F$ to the representation $\sigma$.
These two groups are related by the natural restriction map 
\begin{equation*}
    \text{Sel}_{\text{BK}}(A/F) \to \text{Sel}_{\text{Gr}}(A/F_\infty)^\Gamma
\end{equation*}
with finite kernel and co-kernel. 
If we assume that $\text{Sel}_{\text{BK}}(A/F)$ is finite, then $\text{Sel}_{\text{Gr}}(A/F_\infty)$ becomes a co-torsion $\Lambda$-module. 
Let $\mathfrak{F}$ be a generator of the characteristic ideal of $\text{Sel}_{\text{Gr}}(A/F_\infty)^\vee$, then Longo-Vigni\cite[Theorem 6.1]{LV} proved the following result under certain assumptions \cite[Assumption 2.1]{LV} on the representation $\rho$. 

\begin{theorem}\normalfont{(Longo-Vigni \cite{LV})}
    If  \normalfont{Sel}$_{\text{BK}}(A/F)$ is finite, then $\mathfrak{F}(0) \neq 0$ and 
\begin{equation*}
   \big (  \OO / \mathfrak{F}(0) \cdot \OO \big ) = \#\text{Sel}_{\text{BK}}(A/F) \cdot \prod_{v \in \Sigma \,\, \\ v \nmid p} c_v(A)
\end{equation*}
where $c_v(A)$ is the $p$-part of the Tamagawa number of $A$ at $v$.
\end{theorem}

In this article, we have obtained a similar relation between the cardinality of the Selmer group  and the characteristic power series for an Artin representation $\rho$ under certain mild hypotheses. We now present our main result in more detail.

Throughout, $p$ denotes a fixed odd rational prime. We consider a finite Galois extension
$K/\mathbb{Q}$ with the Galois group $\Delta := \mathrm{Gal}(K/\mathbb{Q})$, such that $p \nmid [K:\mathbb{Q}]$. Let $\FF$ be a finite extension of $\mathbb{Q}_p$
with a ring of integers $\mathcal{O}$, and let $V$ be a $\FF$-vector space of dimension
$d(\rho)$. We then consider an irreducible Artin representation
\begin{equation*}
    \rho : \Delta \to \mathrm{GL}_{\FF}(V).
\end{equation*}
 
Without loss of generality, we assume $\mu_p \subset K$, where $\mu_p$ is the
group of $p$-th roots of unity; otherwise, we replace $K$ by $K' = K(\mu_p)$. We fix an embedding $\iota_\infty :
\overline{\mathbb{Q}} \hookrightarrow \mathbb{C}$ of a fixed algebraic closure
$\overline{\mathbb{Q}}$ of $\mathbb{Q}$ into $\mathbb{C}$, along with an embedding
$\iota_p : \overline{\mathbb{Q}} \hookrightarrow \overline{\mathbb{Q}}_p$ into a fixed
algebraic closure $\overline{\mathbb{Q}}_p$ of $\mathbb{Q}_p$.

Let $\nu$ be an Archimedean place of $K$, and let $K_\nu$ be the corresponding completion. Then $K_\nu$ is isomorphic to  $\mathbb{R}$ or $\mathbb{C}$. We write $\Delta_\nu := \mathrm{Gal}(K_\nu/\mathbb{R})$, which is a subgroup of $\Delta$. Let $d_\nu^+(\rho)$ denote the multiplicity of the trivial representation in the restriction $\rho|_{\Delta_\nu}$. This number does not depend on the choice of $\nu$, so we simply write $d^+(\rho)$. We then define $d^-(\rho) := d(\rho) - d^+(\rho)$. The case $d^+(\rho) = 1$ was studied in \cite{GV} in the context of $p$-adic $L$-functions. In \cite{MajumdarPanda2024}, the algebraic functional equation was studied for the case $1 \le d^+(\rho) \le d(\rho) - 1$. In this article, we focus on the associated Euler characteristics under the hypothesis $d^+(\rho) = 1$.

The embedding $\iota_p$ induces a prime $\wp$ of $K$ above $p$, and we write $\Delta_\wp$
for the decomposition group at $\wp$. Throughout this article, we fix an odd prime $p$ such
that $\mathbb{Q}(\mu_p) \subset K$ is a finite Galois extension of $\mathbb{Q}$ with
$p \nmid |\Delta|$. We work under the following assumptions on the irreducible representation $\rho : \Delta \to \mathrm{GL}_{\FF}(V)$, which are similar to those considered by Greenberg and Vatsal \cite{GV}.

\noindent\textbf{Assumption 1.}
The degree $[K:\Q]$ is prime to $p$ and $d^+(\rho)=1$. Moreover, there exists a one dimensional representation $\varepsilon_\wp$ of $\Delta_\wp$ occurring with multiplicity one in $\rho|_{\Delta_\wp}$. Equivalently, the maximal $\FF$-subspace $V_\wp^+ = V^{(\varepsilon_\wp)}$ of $V$ on which $\Delta_\wp$ acts via the character $\varepsilon_\wp$ is one-dimensional. 

 Throughout this paper, we work with the triple $(\rho,V,V_\wp^+)$ that satisfies the above assumptions. Let $\mathrm{Sel}^\dagger_{\Q_n}(A,A_\wp^+)$, for $\dagger \in \{ \text{g}, \text{GV} \}$, denote the Selmer groups over $\Q_n$ as defined in \cite[Section 2]{MajumdarPanda2024}. Similarly, we write $\mathrm{Sel}^\dagger_{\Q_\infty}(A,A_\wp^+)$ for the corresponding Selmer group over the cyclotomic $\Z_p$-extension $\Q_\infty$. Its Pontryagin dual is denoted by $X^\dagger_{\Q_\infty}(T,T_\wp^+)$. Let $\Gamma := \Gal(\Q_\infty/\Q) \cong \Z_p$, and let $\Lambda_\OO := \OO[[T]]$ be the Iwasawa algebra, identified with the power series ring in one variable over $\OO$. Then $X^\dagger_{\Q_\infty}(T,T_\wp^+)$ carries a natural structure of a $\Lambda_\OO$-module, and is finitely generated over $\Lambda_\OO$. In our situation, the roles of {Sel}$_{\text{BK}}(A/F)$ and $\text{Sel}_{\text{Gr}}(A/F_\infty)$  in \cite{LV} will be replaced by Sel$^{\text{g}}_{\Q}(A,A_\wp^+)$ and Sel$^{\text{GV}}_{\Q_\infty}(A,A_\wp^+)$, respectively. There is a natural restriction map Sel$^{\text{g}}_{\Q}(A,A_\wp^+) \to $ Sel$^{\text{GV}}_{\Q_\infty}(A,A_\wp^+)^\Gamma$ with finite kernel and cokernel under certain assumptions. It follows from \cite[Theorem~1 and Proposition~4.2]{GV} that $\mathrm{Sel}^{\mathrm{g}}_{\Q}(A,A_\wp^+)$ is finite and $\mathrm{Sel}^{\mathrm{GV}}_{\Q_\infty}(A,A_\wp^+)$ is a cofinitely generated cotorsion $\Lambda_{\OO}$-module. Let $\mathfrak{F}$ denote the characteristic power series of
$\mathrm{Sel}^{\mathrm{GV}}_{\mathbb{Q}_\infty}(A,A_\wp^+)^\vee$.
Our main result is the following.

\noindent \begin{thmA} \normalfont{(Theorem \ref{maintheorem})}
\label{A}
\begin{equation*}
  \# \bigl (  \OO / \mathfrak{F}(0) \cdot \OO \bigr ) = \#\text{Sel}^{\text{\,\normalfont{g}}}(A,A_\wp^+) \cdot \prod_{\substack {v \in \Sigma,\, v \nmid p}} c_v(A)
\end{equation*}
\vspace{-0.3cm}
where $c_v(A)$ is the Tamagawa number of $A$ at $v$.
\end{thmA}

In our case, all groups are $p$-primary, and therefore there is no need to separate the $p$-part of the Tamagawa numbers of $A$. Although the methods of Longo–Vigni apply to a broad class of p-adic Galois representations, they cannot be applied directly in the setting of Artin representations considered here. Consequently, some of the arguments need to be adapted to our setting, and we develop these modifications in this paper.


\section{Definition of Selmer Groups}
Let $(\rho, V, V_\wp^+)$ be an irreducible Artin representation. Fix a $\Delta$-invariant $\OO$-lattice $T \subset V$. Since $V_\wp^+$ is $\Delta_\wp$-stable, it induces a filtration $0 \subset T_\wp^+ \subset T,$ where $T_\wp^+ := T \cap V_\wp^+$. Define $V_\wp^- := V/V_\wp^+,\,T_\wp^- := T/T_\wp^+.$ Let $A:=V/T$, and let $A_\wp^+$ denote the image of $V_\wp^+$ in $A$. We further define $A_\wp^- := A/A_\wp^+.$ Then $A$, $A_\wp^+$, and $A_\wp^-$ are discrete $\OO$-modules and are isomorphic, as groups, to $(\FF/\OO)^d$, $(\FF/\OO)^{d^+}$, and $(\FF/\OO)^{d^-}$, respectively. Following Greenberg-Vatsal \cite[p. 285]{GV}, we define $T^*:=\Hom(A,\mu_{p^\infty}),$ which is a free $\OO$-module of rank $d$. Let $V^*:=T^*\otimes_{\OO}\FF,
$ and $A^*:=V^*/T^*.$ Note that the action of $G_{\Q}$ on $V^*$ is given by the contragredient representation twisted by the cyclotomic character. In particular, $V^*$ is  not an Artin representation.

Let $\Sigma$ be a finite set of primes of $\Q$ that contain $p$, the Archimedean place $\infty$, and all primes that ramify in the extension $K/\Q$. Let $\Q_\Sigma$ denote the maximal extension of $\Q$ that is unramified outside $\Sigma$, so in particular $K \subset \Q_\Sigma$. Let $\Q_\infty/\Q$ be the cyclotomic $\Z_p$-extension, and write $\Gamma := \Gal(\Q_\infty/\Q)$. For each integer $n \ge 0$, denote by $\Q_n$ the unique subextension of $\Q_\infty$ with $[\Q_n : \Q] = p^n$. The prime $p$ is totally ramified in both $\Q_n$ and $\Q_\infty$. We denote by $\p_n$ (respectively, $\p_\infty$) the unique prime of $\Q_n$ (respectively, $\Q_\infty$) lying above $p$. If $\omega$ is a prime of $\Q_n$, we write $\omega \mid \Sigma$ to indicate that $\omega$ lies above a prime in $\Sigma$. For each $n \ge 0$, let $v_n$ be a prime of $\Q_n$ lying above a prime in $\Sigma$, and let $v_\infty$ be a prime of $\Q_\infty$ lying above $v_n$. Choose a prime $\overline{v}$ of $\overline{\Q}$ that lies above $v_\infty$ (and therefore above $v_n$). We denote by $G_{v_n}$ (respectively, $G_{v_\infty}$) the decomposition subgroup of $\Gal(\overline{\Q}/\Q_n)$ (respectively, $\Gal(\overline{\Q}/\Q_\infty)$) corresponding to $\overline{v}$. The associated inertia subgroups are denoted by $I_{v_n}$ and $I_{v_\infty}$, respectively.

Following \cite[Introduction]{GV}, we first introduce the local conditions that define the Selmer group.
\begin{equation*}
H^1_{\text{GV}}(G_\omega,A) :=
\begin{cases}
\ker\big(H^1(G_\omega,A) \to H^1(I_\omega,A)\big), & \text{if } \omega \nmid p, \\[6pt]
\ker\big(H^1(G_{\p_n}, A) \to H^1(I_{\p_n}, A_\wp^-)\big), & \text{if } \omega = \p_n.
\end{cases}
\end{equation*}
\begin{definition}\label{selgv}
The Selmer group Sel$^{\text{GV}}_{\Q_n}(A,A_\wp^+)$ is defined by
\begin{equation*}
\mathrm{Sel}^{\text{GV}}_{\Q_n}(A,A_\wp^+) := \ker\Big(H^1(\Q_\Sigma/\Q_n, A) \longrightarrow 
\bigoplus_{\omega \mid \Sigma} 
\frac{H^1(G_\omega, A)}{H^1_{\text{GV}}(G_\omega,A)} \Big).
\end{equation*}
\end{definition}

Equivalently, the Selmer group Sel$^{\text{GV}}_{\Q_n}(A,A_\wp^+)$ can be described as
\begin{equation*}
\mathrm{Sel}^{\text{GV}}_{\Q_n}(A,A_\wp^+) = \ker\Big(H^1(\Q_\Sigma/\Q_n, A) \longrightarrow 
\bigoplus_{\substack{\omega \mid \Sigma,\, \omega \nmid p\infty}} H^1(I_\omega , A)
\;\oplus\;
H^1(I_{\p_n}, A_\wp^-)\Big).    
\end{equation*}
Following Definition~\ref{selgv}, we now define the Selmer group associated to the representation $(\rho, V, V_{\wp}^+)$ over $\Q_n$ as a subgroup of $H^1(\Q_\Sigma/\Q_n, V)$ using a Selmer structure (see \cite[Definition 2.2]{Bellaiche_BK}).  We first specify the local conditions. For each place $\omega$, set
\begin{equation*}
\mathcal{L}_\omega := H^1_\text{g}(G_\omega,V) =
\begin{cases}
\ker\big(H^1(G_\omega,V) \to H^1(I_\omega,V)\big), & \text{if } \omega \nmid p, \\[6pt]
\ker\big(H^1(G_{\p_n}, V) \to H^1(I_{\p_n}, V_\wp^-)\big), & \text{if } \omega = \p_n.
\end{cases}
\end{equation*}
A Selmer structure in $V$ induces one in $A$ via $\text{pr}: V \to A$, giving a map of cohomology.
\begin{equation*}
\text{pr}^*: H^1(G_\omega,V) \longrightarrow H^1(G_\omega,A).    
\end{equation*}
We define $H^1_\text{g}(G_\omega,A)$ as the image of $H^1_\text{g}(G_\omega,V)$ under $\text{pr}^*$, for each place $\omega$ of $\Q_n$ lying above a prime in $\Sigma$. Using these local conditions, we now define an auxiliary Selmer group of $A$ over $\Q_n$ as follows.
\begin{definition}\label{selg}
The Selmer group Sel$^\text{g}_{\Q_n}(A,A_\wp^+)$ is defined by
\begin{equation*}
\ker\Big(H^1(\Q_\Sigma/\Q_n, A) \longrightarrow 
\bigoplus_{\omega \mid \Sigma} 
\frac{H^1(G_\omega, A)}{H^1_\text{g}(G_\omega,A)} \Big).    
\end{equation*}
\end{definition}
For $\dagger \in \{\text{g}, \text{GV}\}$, we define the Pontryagin dual of the Selmer group Sel$^\dagger_{\Q_n}(A,A_\wp^+)$ by
\begin{equation*}
X^\dagger_{\Q_n}(T,T_\wp^+) := \Hom_{\mathrm{cont}}\left( \mathrm{Sel}^\dagger_{\Q_n}(A,A_\wp^+), \Q_p/\Z_p\right).    
\end{equation*}
We now extend these definitions to the cyclotomic $\Z_p$-extension  $\Q_\infty/\Q$. The Selmer group over $\Q_\infty$ is defined by
\begin{equation*}
    \mathrm{Sel}^\dagger_{\Q_\infty}(A,A_\wp^+) := \varinjlim_n \mathrm{Sel}^\dagger_{\Q_n}(A,A_\wp^+),
\end{equation*}
and its Pontryagin dual by
\begin{equation*}
   X^\dagger_{\Q_\infty}(T,T_\wp^+) := \varprojlim_n X^\dagger_{\Q_n}(T,T_\wp^+), 
\end{equation*}
where the limits are taken with respect to the natural restriction and corestriction maps. The natural action of $\Gamma = \Gal(\Q_\infty/\Q)$ on $X^\dagger_{\Q_\infty}(T,T_\wp^+)$ induces a $\Lambda_\OO$-module structure, and $X^\dagger_{\Q_\infty}(T,T_\wp^+)$ is finitely generated over
$\Lambda_{\mathcal O}$ (\cite[p.~226, Corollary]{BalisterHowson1997}).

To compare the two Selmer groups defined above, we look at the local conditions in each prime $\omega \in \Sigma$. For primes $\omega \nmid p$, the difference between local conditions $H^1_{\text{GV}}(G_\omega, A)$ and $H^1_\text{g}(G_\omega, A)$ is captured by the Tamagawa numbers, which we now define.
Consider the following commutative diagram
\begin{equation*}
 \begin{tikzcd}
H^1(G_\omega, V) \arrow[r, "\mathrm{res}_V"] \arrow[d, "\text{pr}^*"'] 
& H^1(I_\omega, V) \arrow[d, "\text{pr}^*"] \\
H^1(G_\omega, A) \arrow[r, "\mathrm{res}_A"'] 
& H^1(I_\omega, A)
\end{tikzcd}   
\end{equation*}

Let $c \in H^1_{\text{g}}(G_\omega, A)$. By definition, there exists $\tilde{c} \in H^1(G_\omega, V)$ such that $\mathrm{pr}^*(\tilde{c}) = c$ and $\mathrm{res}_V(\tilde{c}) = 0$. By commutativity of the diagram, we have $\mathrm{res}_A(c) = \mathrm{pr}^*(\mathrm{res}_V(\tilde{c})) = 0$. Hence $c \in H^1_{\text{GV}}(G_\omega, A)$, and therefore $H^1_{\text{g}}(G_\omega, A) \subset H^1_{\text{GV}}(G_\omega, A)$. 

\begin{lemma} \label{0.5}
\normalfont The subgroup $H^1_{\mathrm{g}}(G_\omega, A)$ has finite index in $H^1_{\mathrm{GV}}(G_\omega, A)$.
\end{lemma}

\begin{proof}
The statement follows from \cite[Lemma 1.3.5 (ii)]{Rubin2000}.
\end{proof}

\begin{lemma} 
\normalfont If $\omega \notin \Sigma$, then $H^1_{\mathrm{g}}(G_\omega, A)$ is same as $H^1_{\mathrm{GV}}(G_\omega, A)$.
\end{lemma}

\begin{proof}
This follows from \cite[Lemma 1.3.5 (iv)]{Rubin2000}.
\end{proof}

Next, we consider the following commutative diagram at the prime $\omega = \p_n$:
\begin{equation*}
 \begin{tikzcd}
H^1(G_{\p_n}, V) \arrow[r, "\mathrm{res}_V"] \arrow[d, "\text{pr}_1^*"'] 
& H^1(I_{\p_n},  V_\wp^-) \arrow[d, "\text{pr}_2^*"] \\
H^1(G_{\p_n}, A) \arrow[r, "\mathrm{res}_A"'] 
& H^1(I_{\p_n},  A_\wp^-)
\end{tikzcd}   
\end{equation*}

Using a similar argument, we obtain $H^1_{\text{g}}(G_{\p_n}, A) \subset H^1_{\text{GV}}(G_{\p_n}, A)$. In fact, one can show that both are equal under certain mild hypothesis.

\begin{lemma} \label{0.6}
\normalfont Assume that $H^0(I_{\mathfrak p_n},A_\wp^-)=0$. Then, if $\omega=\mathfrak p_n$, we have $ H^1_{\mathrm g}(G_{\mathfrak p_n},A) = H^1_{\mathrm{GV}}(G_{\mathfrak p_n},A)$.
\end{lemma}
\begin{proof}
 We already know $ H^1_{\text{g}}(G_{\mathfrak{p}_n},A)\subseteq H^1_{\text{GV}}(G_{\mathfrak{p}_n},A)$. Now take any $ c \in H^1_{\text{GV}}(G_{\mathfrak{p}_n},A)$, so that $
\operatorname{res}_A(c)=0$ in  $H^1(I_{\mathfrak{p}_n},A^-_\wp)$. We want to show that $c \in H^1_{\text{g}}(G_{\mathfrak{p}_n},A)$.

Consider the short exact sequence $0 \to T \to V \to A \to 0$, which induces the exact sequence
\begin{equation*}
 H^1(G_{\mathfrak p_n},V)
\xrightarrow{\mathrm{pr}_1^*}
H^1(G_{\mathfrak p_n},A)
\longrightarrow
H^2(G_{\mathfrak p_n},T).
\end{equation*}
By local Tate duality\cite[Theorem~7.2.6]{NSW2008}, we have $H^2(G_{\mathfrak p_n},T) \cong H^0(G_{\mathfrak p_n},T^*(1))^\vee$. Now $T$ comes from an Artin representation, hence the action of
$G_{\mathfrak p_n}$ on $T$ factors through a finite quotient.
After twisting by the cyclotomic character, the action on
$(T)^*(1)$ is given by a finite-order character times the cyclotomic character.
However, the cyclotomic character has an infinite order. Hence $(T)^*(1)$ cannot contain a nonzero
$G_{\mathfrak p_n}$-invariant element, and hence $(T^*(1))^{G_{\p_n}} = 0$. It follows that $H^2(G_{\p_n}, T) = 0$. Thus, the map $\mathrm{pr}_1^*$ is surjective and therefore, there exists $\tilde{c}\in H^1(G_{\mathfrak{p}_n},V)$ such that $\mathrm{pr}_1^*(\tilde{c})=c$. By the commutativity of the above diagram, we get 
\begin{equation*}
\mathrm{pr}_2^*\bigl(\mathrm{res}_V(\tilde{c})\bigr)
=
\mathrm{res}_A\bigl(\mathrm{pr}_1^*(\tilde{c})\bigr)
=
\mathrm{res}_A(c)
=
0
\in
H^1(I_{\mathfrak{p}_n},A^-_\wp).
\end{equation*}
So, $\mathrm{res}_V(\tilde{c})
\in
\ker\!\left(
\mathrm{pr}_2^* :
H^1(I_{\mathfrak{p}_n},V^-_\wp)
\to
H^1(I_{\mathfrak{p}_n},A^-_\wp)
\right)$.

The short exact sequence  $ 0 \to T^-_\wp \to V^-_\wp \to A^-_\wp \to 0$, induces the long exact sequence
\begin{equation*}
H^0(I_{\mathfrak{p}_n},A^-_\wp) \xrightarrow{\ \delta  \ } H^1(I_{\mathfrak{p}_n},T^-_\wp)
\xrightarrow{\ \iota_1^*\ }
H^1(I_{\mathfrak{p}_n},V^-_\wp)
\xrightarrow{\ \mathrm{pr}_2^*\ }
H^1(I_{\mathfrak{p}_n},A^-_\wp).  
\end{equation*}
Therefore, $ \ker\!\left(\mathrm{pr}_2^* :H^1(I_{\mathfrak{p}_n},V^-_\wp) \to H^1(I_{\mathfrak{p}_n},A^-_\wp)\right) = \iota_*\!\left(H^1(I_{\mathfrak{p}_n},T^-_\wp)\right).$ Hence, $\mathrm{res}_V(\tilde{c}) = \iota_1(\tilde{s})$ for some $\tilde{s} \in H^1(I_{\mathfrak{p}_n},T^-_\wp).$ We can show that  $\tilde{s} \in H^1(I_{\mathfrak p_n},T^-_\wp)^{G_{\mathfrak p_n}/I_{\mathfrak p_n}}$. Since the natural map $T^-_\wp \to V^-_\wp$ is $G_{\mathfrak{p}_n}$ equivariant, the induced map $\iota_1^*$  commutes with the $G_{\mathfrak p_n}/I_{\mathfrak p_n}$-action. Therefore 
\begin{equation*}
 \iota_1^*(\sigma \cdot \tilde{s}) = \sigma \cdot \iota_1^*(\tilde{s}) = \sigma \cdot  \mathrm{res}_V(\tilde{c}) = \mathrm{res}_V(\tilde{c}) = \iota_1^*(\tilde{s}),
\end{equation*}
we obtain $\sigma \cdot \tilde{s} - \tilde{s} \in \text{Ker}(\iota_1^*)$. Since $H^0(I_{\mathfrak{p}_n},A^-_\wp)=0$, it follows that $\iota_1^*$ is injective. Therefore
$\sigma \cdot \tilde{s} = \tilde{s}$ for all $\sigma \in G_{\mathfrak p_n}/I_{\mathfrak p_n}$.

Since $\tilde{s} \in H^1(I_{\mathfrak p_n},T^-_\wp)^{G_{\mathfrak p_n}/I_{\mathfrak p_n}}$, we first lift $\tilde{s}$ to a class in $\tilde{s}$ in $H^1(G_{\mathfrak p_n},T^-_\wp)$, and then we lift it further to a class in $H^1(G_{\mathfrak p_n},T)$. To obtain the second lift, consider the short exact sequence $0 \to T^+_\wp \to T \to T^-_\wp \to 0$. The associated long exact sequence gives 
\begin{equation*}
H^1(G_{\mathfrak p_n},T)
\xrightarrow{\ \mathrm{pr}_3^*\ }
H^1(G_{\mathfrak p_n},T^-_\wp)
\longrightarrow
H^2(G_{\mathfrak p_n},T^+_\wp).    
\end{equation*}
Once again, by local Tate duality, $H^2(G_{\mathfrak p_n},T^+_\wp) \cong H^0(G_{\mathfrak p_n},(T^+_\wp)^*(1))^\vee$. Using the same argument as before, we can show that $(T^+_\wp)^*(1)$ has no nonzero element invariant under the  $G_{\mathfrak p_n}$-action. So $H^0(G_{\mathfrak p_n},(T^+_\wp)^*(1))=0$ and therefore, $H^2(G_{\mathfrak p_n},T^+_\wp)=0$ and the map $\mathrm{pr}_3^*$  is surjective. 
\vspace{0.1cm}
Next, consider the inflation--restriction sequence 
{
\[
 H^1(G_{\mathfrak p_n},T^-_\wp)
\xrightarrow{\ \mathrm{res}_{T^-_\wp}\ }
H^1(I_{\mathfrak p_n},T^-_\wp)^{G_{\mathfrak p_n}/I_{\mathfrak p_n}}
\to H^2(G_{\mathfrak p_n}/I_{\mathfrak p_n},(T^-_\wp)^{I_{\mathfrak p_n}}).
\]
} Since the cohomological dimension $G_{\mathfrak p_n}/I_{\mathfrak p_n}$ is  $1$, we have $H^2(G_{\mathfrak p_n}/I_{\mathfrak p_n},(T^-_\wp)^{I_{\mathfrak p_n}})=0$. Hence the map $\mathrm{res_{T^-_\wp}}$ is surjective. Since $\mathrm{pr}_3^*$  is also surjective, we can find $ t \in H^1(G_{\mathfrak{ p_n}},T)$ which lifts $\tilde{s}$.

Now define $c'=\tilde{c}-\iota_2^*(t)\in H^1(G_{\mathfrak p_n},V)
$, where $\iota_2^*$ is the natural map from $H^1(G_{\mathfrak p_n},T)$ to $H^1(G_{\mathfrak p_n},V)$. By commutativity of the restriction maps\cite[Proposition 1.5.3]{NSW2008}, we have 
$\mathrm{res}_V\bigl(\iota_2^*(t)\bigr) 
= \iota_1^*\!\Bigl(\mathrm{res}_{T^-_\wp}\bigl(\mathrm{pr}_3^*(t)\bigr)\Bigr)$. Therefore $\mathrm{res}_V( c')=\mathrm{res}_V(\tilde c)-\mathrm{res}_V\bigl(\iota_2^*(t)\bigr)=\iota_1^*(\tilde s)-\iota_1^*(\tilde s)=0.$ Hence, by definition $c' \in H^1_{\mathrm g}(G_{\mathfrak p_n},V)$. Moreover, since $\mathrm{pr}_1^* \circ \iota_2^*=0$, we have $\mathrm{pr}_1^*(c')=\mathrm{pr}_1^*(\tilde c)-\mathrm{pr}_1^*(\iota_2^*(t)) =c.$ Therefore, $c \in H^1_{\mathrm{g}}(G_{\mathfrak p_n},A)$. This proves $ H^1_{\mathrm{GV}}(G_{\mathfrak{p}_n},A) \subset H^1_{\mathrm{g}}(G_{\mathfrak{p}_n},A)$.

\end{proof}

Using Lemma~\ref{0.5} and \ref{0.6}, we make the following definition.

\begin{definition}
We define the Tamagawa number of $A$ at $\omega$ by
\begin{equation*}
  c_\omega(A) := \big[ H^1_{\mathrm{GV}}(G_\omega, A) : H^1_{\mathrm{g}}(G_\omega, A) \big].  
\end{equation*}
\end{definition}

\begin{proposition}\label{prop:main}

\normalfont There is an exact sequence
\begin{equation*}
    0 \longrightarrow
\mathrm{Sel}^{\mathrm{g}}_{\Q}(A,A_\wp^+)
\longrightarrow
\mathrm{Sel}^{\mathrm{GV}}_{\Q}(A,A_\wp^+)
\xlongrightarrow{}
\bigoplus_{\substack{\omega\in\Sigma\\ \omega\nmid p}}
\frac{H^1_{\mathrm{GV}}(G_\omega,A)}
     {H^1_{\mathrm{g}}(G_\omega,A)}.
\end{equation*}
In particular, we have the following result.
\begin{equation*}
    \frac{\#\mathrm{Sel}^{\mathrm{GV}}_{\Q}(A,A_\wp^+)}
     {\#\mathrm{Sel}^{\mathrm{g}}_{\Q}(A,A_\wp^+)}
\;\Big|\;
\prod_{\substack{\omega\in\Sigma\\ \omega\nmid p}}
c_\omega(A).
\end{equation*}
\end{proposition}

\begin{proof}
The proof of the exact sequence follows from Lemma~\ref{0.5}, which gives $H^1_{\mathrm{g}}(G_\omega,A)\subset H^1_{\mathrm{GV}}(G_\omega,A)$ for every $\omega \in \Sigma$ and from Lemma~\ref{0.6}, which states that $H^1_{\mathrm{g}}(G_{\mathfrak p_n},A)=H^1_{\mathrm{GV}}(G_{\mathfrak p_n},A)$. The divisibility statement follows by taking orders in the above exact sequence and using the definition of the Tamagawa numbers.
  
\end{proof}

\section{Control Theorem}
In this section, we discuss the behavior of the Selmer groups $\mathrm{Sel}^{\mathrm{g}}_{\Q}(A,A_\wp^+)$ in the cyclotomic tower, which relates $\mathrm{Sel}^{\mathrm{g}}_{\Q}(A,A_\wp^+)$ 
to the $\Gamma$-invariants of $\mathrm{Sel}^{\mathrm{GV}}_{\Q}(A,A_\wp^+)$. This is analogous of the control theorem of Mazur  for elliptic curves\cite{MA}, which relates $\text{Sel}_p(E/F)$ to $\text{Sel}_p(E/F_\infty)^\Gamma$  and to the control theorem of Longo–Vigni for $p$-adic Galois representations, which relates $\text{Sel}_{\text{BK}}(A/F)$  to $\text{Sel}_{\text{Gr}}(A/F_\infty)^\Gamma$. We provide an analogous statement for Artin representations. Before proving the control theorem, we state a few results that will be needed later.

\begin{lemma} \label{3.1}
\normalfont If $V^{G_{\Q}}=0$, then $A^{G_{\Q}}=0$.
\end{lemma}

\begin{proof}
Since the action factors through the finite group $\Delta = \mathrm{Gal}(K/\Q)$, we have $H^0(\Delta, V) = V^{G_{\Q}}$. Consider the short exact sequence $0 \to T \to V \to A \to 0$, which induces the exact sequence
\begin{equation*}
H^0(\Delta, V)
\longrightarrow
H^0(\Delta, A)
\longrightarrow
H^1(\Delta, T).    
\end{equation*}
Since $\Delta$ is finite and $p \nmid \Delta$,  Proposition~1.6.2 of \cite{NSW2008} implies that $H^1(\Delta, T)=0$. Therefore, $H^0(\Delta, A)=0$, and hence $A^{G_{\Q}}=0$.

\end{proof}

\begin{lemma}\label{3.2}
\normalfont If $(\rho,V)$ is a non-trivial irreducible Artin representation, then $H^0(G_{\Q}, A^*) = 0$.    
\end{lemma}

\begin{proof}
From \cite[p.~285]{GV}, the dual representation $V^*$ can be represented as $V^* \cong U \otimes \kappa$, where $U$ is an Artin representation and $\kappa$ factors through
$\Gamma=\mathrm{Gal}(\Q_\infty/\Q)$. Hence $(V^*)^{\Delta}= U^{\Delta}$. Since $V$ is irreducible, $U$ is also irreducible. Hence, $U^\Delta = 0$, and consequently $(V^*)^\Delta=0$. Since the action of $V^*$ also factors through $\Delta$, it follows that $(V^*)^{G_{\Q}}=0$. By Lemma~\ref{3.1}, we have $H^0(G_{\Q}, A^*) = 0$.   
\end{proof}

\begin{theorem}\label{3.3}
 \normalfont The restriction map
 \begin{equation*}
 r:\mathrm{Sel}^{\mathrm{g}}_{\Q}(A,A_\wp^+)\longrightarrow
\mathrm{Sel}^{\mathrm{GV}}_{\Q_\infty}(A,A_\wp^+)^\Gamma    
 \end{equation*}
is injective and has finite co-kernel.
\end{theorem}\label{theorem:main}

\begin{proof}
Consider the following inflation-restriction sequence
\begin{equation*}
0 \to H^1(\Gamma,A^{G_{\Q_\infty}}) \to H^1(G_{\Q,\Sigma},A) \xrightarrow{\mathrm{res}} H^1(G_{\Q_\infty,\Sigma},A)^\Gamma \to H^2(\Gamma,A^{G_{\Q_\infty}}).    
\end{equation*}

Since $\rho$ is non-trivial and irreducible, we have $V^{G_{\Q}}=0$. Therefore, $A^{G_{\Q}}=0$ by Lemma~\ref{3.1}. Let $K_n = K\Q_n$ be the $n$-th layer of the cyclotomic $\mathbf{Z}_p$-extension $K_\infty/K$. Since $p \nmid |\Delta|$, we have $\Delta \simeq \mathrm{Gal}(K_n/\Q_n)$.
 Now, for any $n$, the action of $G_{\Q_n}$ factors through
$\mathrm{Gal}(K_n/\Q_n)$, and hence $A^{G_{\Q_n}} =A^{\mathrm{Gal}(K_n/\Q_n)} = A^\Delta=A^{G_\Q}.$ Passing the limit over $n$, we obtain $A^{G_{\Q_\infty}} = A^{G_\Q} = 0$. Therefore, both $H^1(\Gamma,A^{G_{\Q_\infty}})$ and $H^2(\Gamma,A^{G_{\Q_\infty}})$ are zero, and from the restriction map we get an isomorphism $H^1(G_{\Q,\Sigma},A)  \cong H^1(G_{\Q_\infty,\Sigma},A)^\Gamma$. Now, consider the commutative diagram

\begin{tikzcd}[column sep=2.2em,row sep=3em]
0 \ar[r] &
\mathrm{Sel}^{\mathrm{g}}_{\Q}(A,A_\wp^+) \ar[r] \ar[d,"r"] &
H^1(G_{\Q,\Sigma},A) \ar[r,"\lambda"] \ar[d,"s"] &
{\bigoplus\limits_{\omega\in\Sigma}
\frac{H^1(G_\omega,A)}
     {H^1_{\mathrm{g}}(G_\omega,A)}}
\ar[d,"\phi"] 
\\
0 \ar[r] &
\mathrm{Sel}^{\mathrm{GV}}_{\Q_\infty}(A,A_\wp^+)^\Gamma \ar[r] &
H^1(G_{\Q_\infty,\Sigma},A)^\Gamma \ar[r,"\mu"] &
{\mathop{\bigoplus}\limits_{\omega\in\Sigma}
\left(
\frac{H^1(G_{\omega,\infty},A)}
     {H^1_{\mathrm{GV}}(G_{\omega,\infty},A)}
\right)^\Gamma}
\end{tikzcd}

 The snake lemma \cite[Theorem 1]{lemmermeyer2011snakelemma} gives an exact sequence
\begin{equation*}
 \ker(s)
\longrightarrow \ker(\phi) \cap \operatorname{img}(\lambda)
\longrightarrow \operatorname{coker}(r)
\longrightarrow \operatorname{coker}(s)
\end{equation*}
On the other hand, both $\ker(s)$ and $\operatorname{coker}(s)$ are trivial, so the exact sequence gives an isomorphism $\operatorname{coker}(r) \cong   \ker(\phi) \cap \operatorname{img}(\lambda)$. The kernel of $\phi$ is $\mathop{\bigoplus}\limits_{\substack{\omega\in\Sigma,\omega\nmid p}}
\frac{H^1_{\mathrm{GV}}(G_\omega,A)}
     {H^1_\mathrm{g}(G_\omega,A)}$, which is finite of order $\mathop{\prod}\limits_{\substack{\omega\in\Sigma,\,\omega\nmid p}}
c_\omega(A)$. Hence $\operatorname{coker}(r)$ is finite. Next, since $s$ is injective, it is clear that $r$ is also injective. This completes the proof of the theorem.
\end{proof}

\begin{lemma} \label{surjective}
\normalfont The map $\lambda$ is surjective.    
\end{lemma}

\begin{proof}
We use \cite[Proposition~3.2.1]{Greenberg2010} to show the surjectivity of $\lambda$.  First, since $A \cong (\FF/\OO)^d$,
the module $A$ is divisible as a $\mathbf \OO$-module. Furthermore, by \cite[Proposition 2.2 and Corollary 2.7]{GV}, we have
\begin{equation*}
H^2(\Q_\Sigma/\Q,A)=0.    
\end{equation*}
Hence the weak Leopoldt condition LEO(A) is satisfied. Next, \cite[Proposition 4.2]{GV} shows that the cokernel of the global-to-local map $\phi_n$, which is defined using the local conditions $H^1_{\text{GV}}(G_\omega,A)$ is a cotorsion $\Lambda_\OO$-module. On the other hand, the global-to-local map $\lambda$ is defined using local conditions $H^1_{\text{g}}(G_\omega,A)$. For every place $\omega \in \Sigma$, the subgroup $H^1_{\text{g}}(G_\omega,A)$ has finite index in $H^1_{\text{GV}}(G_\omega,A)$ by Lemmas \ref{0.5} and \ref{0.6}. Therefore, the quotients $H^1(G_\omega,A)/H^1_{\mathrm{g}}(G_\omega,A)$ and $H^1(G_\omega,A)/H^1_{\mathrm{GV}}(G_\omega,A)$ differ by a finite group. Since $\Sigma$ is finite, the targets of $\lambda$ and $\phi_n$  differ by a finite group. Consequently, their cokernels differ by a finite group. Since the cokernel of $\phi_n$ is a cotorsion $\Lambda_\OO$-module\cite[Proposition 4.2]{GV}, the same is true for the cokernel of $\lambda$. Thus hypotheses $(1)$ and $(2)$ of \cite[Proposition~3.2.1]{Greenberg2010} are satisfied.
 
Finally, since $p \nmid |\Delta|$, the reduction of $\rho$ modulo $(\pi)$ remains irreducible over $\OO/\pi\OO$. 

Hence $A[\pi] \cong T/\pi T$ is irreducible as a $G_{\Q}$-module. In particular, it has no subquotient isomorphic to $\mu_p$. Hence \cite[Proposition~3.2.1]{Greenberg2010} implies that $\lambda$ is surjective.
\end{proof}

\section{Euler Characteristic Formula}
In this section, we prove Theorem A. Before giving the proof, we first establish several auxiliary lemmas. To simplify the notation, let $
S:=\mathrm{Sel}^{\mathrm{GV}}_{\Q_\infty}(A,A_\wp^+)$ and $
X:= S^\vee = X^{\mathrm{GV}}_{\Q_\infty}(T,T_\wp^+)$.

\begin{lemma} \label{4.1}
\normalfont
Let $M$ be a cofinitely generated cotorsion $\Lambda_\OO$-module. Let $f(T)$ be a generator of the characteristic ideal of $M^\vee$. Assume that $M^\Gamma$ is finite. Then $M_{\Gamma}$ is finite, $f(0) \neq 0$ and 
\begin{equation*}
 \# \bigl (  \OO / f(0) \cdot \OO \bigr ) =   \frac{\# M^\Gamma}{\# M_\Gamma}.
\end{equation*}
\end{lemma}

\begin{proof}
The result follows from \cite[Lemma~4.2]{GR5}.
\end{proof}

Note that $\Gamma$ has cohomological dimension one and hence $H^i(\Gamma, M) = 0$ for all $i \geq 2$. Therefore, $\frac{\# M^\Gamma}{\# M_\Gamma}$ is equal to the Euler characteristic $\chi(\Gamma,M) =
\# H^0(\Gamma, M)/ \# H^1(\Gamma, M)$. 



By \cite[Proposition~4.2]{GV}, $S$ is a cofinitely generated cotorsion $\Lambda_\OO$-module. Let $\mathfrak{F}$ denote its characteristic power series of $X$. Since $\mathrm{Sel}^{\mathrm{g}}_{\Q}(A,A_\wp^+)$ is finite \cite[Theorem 1]{GV}, it follows from Theorem \ref{3.3} that $S^\Gamma$ is also finite. Now, from Lemma~\ref{4.1} , we obtain

\begin{equation} \label{main}
 \# \bigl (  \OO / \mathfrak{F}(0) \cdot \OO \bigr ) =   \frac{\# S^\Gamma}{\# S_\Gamma}.
\end{equation}

\begin{lemma}
\label{4.3}
 $\# S^\Gamma  =  \# \mathrm{Sel}^{\mathrm{g}}_{\Q}(A,A_\wp^+) \cdot \mathop{\prod}\limits_{\substack{\omega\in\Sigma, \, \omega\nmid p}}
c_\omega(A) $
\end{lemma}

\begin{proof}
By Lemma~\ref{surjective}, the map $\lambda$ is surjective. Then it follows from Theorem~\ref{3.3} that $\operatorname{coker}(r)\cong \operatorname{ker}(\phi)$. Since $r$ is injective and $\#\operatorname{coker}(r)
= \displaystyle
\prod_{\substack{\omega\in\Sigma,\, \omega\nmid p}}
c_\omega(A)$, The lemma now 
\vspace{-0.5cm} \\
follows from the exact sequence
\begin{equation*}
0 \longrightarrow \mathrm{Sel}^{\mathrm{g}}_{\Q}(A,A_\wp^+)
\longrightarrow S^\Gamma
\longrightarrow \ker(\phi)
\longrightarrow 0.    
\end{equation*}
\end{proof}

Next, we show that $S_\Gamma=0$, which is crucial for the proof of our main theorem.

\begin{lemma} \label{4.4}
\normalfont $S_\Gamma=0$.
\end{lemma}

\begin{proof}
By \cite[Theorem~1]{GV} and Theorem~\ref{3.3}, the module
$S^\Gamma$ is finite. We have $\left(S^\Gamma\right)^\vee \cong X_\Gamma$. Hence, $X_\Gamma$ is finite. By \cite[Theorem~2]{GV}, $X$ is a finitely generated torsion $\Lambda_{\mathcal O}$-module. Therefore, $\operatorname{rank}_{\Lambda_{\mathcal O}}(X)=0$. Now, by applying \cite[Proposition~5.3.20]{NSW2008}, we obtain
\begin{equation*}
	\operatorname{rank}_{\Lambda_{\mathcal O}}(X)
	=
	\operatorname{rank}_{\mathcal O}(X_\Gamma)
	-
	\operatorname{rank}_{\mathcal O}(X^\Gamma).
\end{equation*}
Since $X_\Gamma$ is finite, its $\mathcal O$-rank is zero. It follows that $\operatorname{rank}_{\mathcal O}(X^\Gamma)=0$. Therefore, $X^\Gamma$ is finite. Next, by \cite[Proposition~4.5]{GV}, $X$ has no nontrivial finite
$\Lambda_{\mathcal O}$-submodules. Since $X^\Gamma$ is a finite
$\Lambda_{\mathcal O}$-submodule of $X$, we must have $X^\Gamma=0$. Finally, we have $\left(X^\Gamma\right)^\vee \cong S_\Gamma$. Therefore, $S_\Gamma=0$.
\end{proof}
We now prove the main result of this section, which relates the value
\(\mathfrak{F}(0)\) to the Selmer group and the Tamagawa factors.
\begin{theorem}
\normalfont \label{maintheorem} Let $(\rho,V,V_\wp^+)$ be a triple that satisfies Assumption-1. Assume that $H^0(I_{\mathfrak p_n},A_\wp^-)=0$. Then we have the following equality
\begin{equation*}
  \# \bigl (  \OO / \mathfrak{F}(0) \cdot \OO \bigr ) = \# \mathrm{Sel}^{\mathrm{g}}_{\Q}(A,A_\wp^+) \cdot \prod_{\substack {v \in \Sigma, \\ v \nmid p}} c_v(A)
\end{equation*}
where $c_v(A)$ is the Tamagawa number of $A$ at $v$.
\end{theorem}

\begin{proof}
By Lemma~\ref{4.4}, we have $S_\Gamma=0$, and hence
$\#S_\Gamma=1$. The theorem now follows from
\eqref{main} and Lemma~\ref{4.3}.     
\end{proof}

\section*{Acknowledgement}
The author would like to thank the anonymous referee for the helpful comments and suggestions that have improved the presentation of the paper.
\bibliographystyle{amsplain}
\bibliography{mybib}

\end{document}